\documentclass[11pt]{article}
\usepackage{latexsym,amsmath,stackrel,color,amsthm,amssymb,epsfig,graphicx,mathrsfs}
\usepackage{xpatch}
\usepackage{graphicx}
\usepackage[title]{appendix}
\usepackage{amssymb}
\usepackage[left=1in,top=1in,right=1in,bottom=1in]{geometry}
\usepackage[linktocpage=true]{hyperref}
\usepackage{setspace}
\usepackage{amssymb, amsmath, amsthm, graphicx,mathrsfs}
\usepackage{caption}
\usepackage{subcaption}
\usepackage{enumitem}
\setlist[itemize]{noitemsep, topsep=0pt}
\usepackage{cite}
\usepackage{titlesec}
\usepackage{thmtools}
\usepackage{lipsum}

\usepackage{comment,caption}
\usepackage{tikz}
\usetikzlibrary{calc,intersections}
\makeatletter
\def\thm@space@setup{%
  \thm@preskip=\parskip \thm@postskip=0pt
}
\makeatother

\makeatletter
\renewenvironment{proof}[1][\proofname]{\par
  \pushQED{\qed}%
  \normalfont \partopsep=\z@skip \topsep=\z@skip
  \trivlist
  \item[\hskip\labelsep
        \itshape
    #1\@addpunct{.}]\ignorespaces
}{%
  \popQED\endtrivlist\@endpefalse
}
\makeatother

\usepackage{hyperref}
\hypersetup{
    colorlinks=true,
    citecolor = blue,
    linkcolor= black,
    filecolor=magenta,      
    urlcolor=cyan,
    pdftitle={Overleaf Example},
    pdfpagemode=FullScreen,
    }

\usepackage[capitalize,nameinlink]{cleveref}
\usepackage{nameref}

\crefname{secinapp}{Appendix}{appendices}
\crefformat{subsection}{\S#2#1#3}
\crefalias{appendixthm}{thm}

\declaretheoremstyle[%
  spaceabove=2pt,%
  spacebelow=0pt,%
  headfont=\normalfont\itshape,%
  postheadspace=1em,%
  qed=\qedsymbol%
]{mystyle}
\newtheorem{thm}{Theorem}
\newtheorem{appendixthm}{Theorem}[section]

\newtheorem{appendixclaim}{Claim}[section]

\newtheorem{const}{Construction}
\newtheorem{problem}{Problem}

\newtheorem{lemma}{Lemma}

\newtheorem{question}{Question}

\newcommand{\Mod}[1]{(\mathrm{mod}\ #1)}

\tikzstyle{vertex}=[circle,fill=black,inner sep=2pt]
\tikzstyle{vertrect}=[draw,rectangle,inner sep=2pt]
\tikzstyle{vertdia}=[draw,diamond,inner sep=2pt]
\newcommand{\vb}[1]{\boldsymbol{#1}}

\newcommand{\N}{\mathbb{N}}
\newcommand{\Z}{\mathbb{Z}}

\newcommand{\A}{\mathcal{A}}
\newcommand{\F}{\mathcal{F}}
\newcommand{\B}{\mathcal{B}}

\newcommand{\Span}{\mathrm{Span}}

\newcommand{\ff}[1]{\mathbb{F}_2^#1}

\titleformat{\subsection}[runin]
        {\normalfont\bfseries}
        {\thesubsection}
        {0.4em}
        {}
        [.]

\begin{document}
\title{A few new oddtown and eventown problems} 
\author{
	Griffin Johnston \\
	Department of Mathematics, \\
	Emory University \\
	\texttt{jgjohn5@emory.edu} 
        \and 
        Jason O'Neill \\
	Department of Mathematics, \\
	California State University, Los Angeles \\
	\texttt{joneill2@calstatela.edu} 
}

\maketitle
\begin{abstract}
Given a vector $\alpha = (\alpha_1, \ldots, \alpha_k) \in \mathbb{F}_2^{k}$, we say a collection of subsets $\mathcal{F}$ satisfies $\alpha$-intersection pattern modulo $2$ if all $i$-wise intersections consisting of $i$ distinct sets from $\mathcal{F}$ have size $\alpha_i \pmod{2}$. In this language, the classical oddtown and eventown problems correspond to vectors $\alpha=(1,0)$ and $\alpha=(0,0)$ respectively. In this paper, we determine the largest such set families of subsets on a $n$-element set  with $\alpha$-intersection pattern modulo $2$ for all $\alpha \in \mathbb{F}_2^{3}$ and all $\alpha \in \mathbb{F}_2^{4}$ asymptotically. Lastly, we consider the corresponding problem with restrictions modulo $3$.  
\end{abstract}

\section{Introduction}
Given a collection $\F$ of subsets of an $n$ element set, $\F$ follows \textit{oddtown rules} if the sizes of all sets in $\F$ are odd and distinct pairs of sets from $\F$ have even sized intersections. Similarly, a collection $\F$ of subsets of an $n$ element set follows \textit{eventown rules} if the sizes of all sets in $\F$ are even and distinct pairs of sets from $\F$ have even sized intersections. Berlekamp \cite{BERK} and Graver\cite{GRA} proved that the size of a family which satisfies oddtown rules is at most $n$ and a family which satisfies eventown is at most $2^{\lfloor n/2 \rfloor}$, and these results are best possible. 

There has been a substantial \cite{GishbolinerTomonSudakov,ONEILL,VVU0,VVU1,DFS} amount of research devoted to generalizations and applications of oddtown and eventown. See the book by Frankl and Tokushige \cite{FT} for an excellent survey of intersection problems for finite sets. Very recently, the breakthrough of Conlon and Ferber \cite{CF}, and subsequent work of Wigderson \cite{WIG}, utilized oddtown and eventown results and randomness to improve lower bounds on mutlicolor Ramsey numbers. (Although Sawin \cite{SAWIN} further improved these bounds without oddtown and eventown.) Further, there has been active research \cite{OV,VVU2,FS,GS} in $k$-wise generalizations of the oddtown and eventown problems, which involve constraints on the $k$-wise intersections from a family. In this paper, we consider another such $k$-wise generalization. 
 
Let $[n]=\{1,2,\ldots,n\}$, $2^{[n]}$ denote the collection of all subsets of $[n]$, and $\binom{[n]}{r}$ denote all size $r$ subsets of $[n]$ for an integer $1 \leq r \leq n$. Let $\alpha = (\alpha_1, \ldots, \alpha_k) \in \ff{k}$. Then, a collection $\F \subset 2^{[n]}$ satisfies {\it $\alpha$-intersection pattern modulo $2$} if $|F_1 \cap \cdots \cap F_i| = \alpha_i \pmod{2}$ for all distinct $F_1, \ldots, F_i \in \F$. Let $f_\alpha(n)$ denote the maximum size family $\F \subset 2^{[n]}$ where $\F$ satisfies $\alpha$-intersection pattern modulo $2$. In this paper, we primarily focus on $f_\alpha(n)$ for $\alpha \in \ff{3}$ and $\alpha \in \ff{4}$ and build upon previous results on $f_\alpha(n)$ in the literature.

In our language, the classical oddtown and eventown problems of Berlekamp \cite{BERK} and Graver\cite{GRA} show that $f_{(1,0)}(n)=n$ and $f_{(0,0)}(n) = 2^{\lfloor n/2 \rfloor}$ respectively. The dual oddtown problem (referred to as the ``Reverse" oddtown problem in \cite[Exercise 1.1.5]{BF}) shows that $f_{(0,1)}(n) = n$ when $n$ is odd and $f_{(0,1)}(n)=n-1$ when $n$ is even. We are unaware if the dual eventown problem (i.e. the case when $\alpha=(1,1)$) appears in the literature and as such include a proof in \cref{sec:appendix}. We show that $f_{(1,1)}(n) = 2^{ \lfloor (n-1)/2 \rfloor }$ and therefore these results collectively handle $f_\alpha(n)$ for all $\alpha \in \ff{2}$.

For vectors $ \alpha \in \ff{3}$, six of the eight vectors follow with a short argument that we include in \cref{sec:3wise}. Our first theorem handles the two remaining vectors in $\ff{3}$:

\begin{thm}\label{thm:alpha001}
Let $n\geq 7$. Then 
\begin{equation*}
 f_{(1,1,0)}(n) = 
    \begin{cases}
        \lfloor n/2 \rfloor + 1 & \text{$n \equiv 2,3 \pmod{4}$} \\
        \lfloor n/2 \rfloor & \text{$n \equiv 0,1 \pmod{4}$}
    \end{cases}    
        \quad 
        \text{and}
        \quad 
 f_{(0,0,1)}(n) = 
    \begin{cases}
        \lfloor n/2 \rfloor + 1 & \text{
        $n \equiv 3 \pmod{4}$} \\
        \lfloor n/2 \rfloor & 
        \text{otherwise}
    \end{cases} \; .        
\end{equation*}
\end{thm}
Theorem \ref{thm:alpha001}, together with a few short arguments in Section 3, will establish the following table: 

{\hypersetup{linkcolor = blue}
\begin{table}[h!]
\centering
{\renewcommand{\arraystretch}{1.5}
\begin{tabular}{|c|c|c||c|c|c|}
\hline
$\alpha$ & section &  $f_{\alpha} (n)$ & $\alpha$ & section & $f_{\alpha} (n)$ \\
\hline
$(1,0,0)$ & \cref{subsec:100} &  $n$ & 
$(0,1,1)$ & \cref{subsec:011} &  $n-1$ \\ 
\hline
$(0,0,0)$ & \cref{subsec:000}  & $2^{\lfloor n/2 \rfloor}$ 
& $(1,1,1)$ & \cref{subsec:111} &  $2^{\lfloor (n-1)/2 \rfloor}$ \\ 
\hline 
$(1,0,1)$ &  \cref{subsec:101} &  
$  \begin{cases} n  & \text{$n$ even} \\ 
n-1  & \text{$n$ odd}
\end{cases} $ 
 & $(0,1,0)$ & \cref{subsec:010} & $ \begin{cases}
            n-1 & \text{$n$ even} \\
            n  & \text{$n$ odd}
        \end{cases} $ \\ 
\hline
$(1,1,0)$ &  \cref{subsec:110}  &$  \begin{cases}
        \lfloor n/2 \rfloor + 1 & \text{$n \equiv 2,3 $} \\
        \lfloor n/2 \rfloor & \text{$n \equiv 0,1$}
    \end{cases}   $ 
& $(0,0,1)$ &  \cref{subsec:001} &  $  \begin{cases}
        \lfloor n/2 \rfloor + 1 & \text{
        $n \equiv 3 $} \\
        \lfloor n/2 \rfloor & 
        \text{otherwise}
    \end{cases}$ \\ 
\hline
\end{tabular}
}
\caption{Our $3$-wise results for $n\geq 7$. For ease of display, the cases in the last row are taken modulo $4$.}
\end{table}
}

For intersection patterns $\alpha \in \ff{k}$ with $k \geq 4$, we will need to utilize asymptotic notation to state our results. For functions $f,g :\mathbb N\rightarrow \mathbb R^+$, we write $f = o(g)$ if $\lim_{n \rightarrow \infty} f(n)/g(n) = 0$, and $f = O(g)$ if there is $c > 0$ such that $f(n)\leq cg(n)$ for all $n \in \mathbb N$. If $f = O(g)$ and $g = O(f)$, we write $f =\Theta(g)$. If $g = O(f)$, we write $f=\Omega(g)$. We write $f \sim g$ if $\lim_{n \rightarrow \infty} f(n)/g(n) = 1$. Sudakov and Vieira \cite{SV} studied the case where $\alpha=(0,0, \ldots, 0)$, which they referred to as {\it Strong $k$-wise eventown} towards proving a nice stability result on $k$-wise eventown when $k \geq 3$ in contrast to the classical eventown problem. The second author and Verstra\"{e}te \cite{OV} showed that $f_\alpha(n) = \Theta (n^{1/t})$ when $\alpha=(1,1,\ldots,1,0,0, \ldots, 0)$ with $t$ leading zeroes and $k-t$ trailing ones so long as $2t-2 \leq k$. It is also worth noting that other $k$-wise works \cite{GS,FS,VVU2} consider problems with restrictions on the sets and the $k$-wise intersections from the family without conditions on the intermediate intersections.

For vectors in $\ff{4}$, we state asymptotic results for eight vectors in total. However, in \cref{sec:prelim}, we establish that $f_{\alpha}(n-1) \leq f_{\alpha + \mathbf{1}}(n) \leq f_{\alpha}(n+1)$, which yields a natural duality between these problems for vectors $\alpha$ and $\alpha+\mathbf{1}$. In fact, this duality is apparent in Table 1 by examining the four different rows in the table.  This, together with our work, establishes asymptotic results for all sixteen vectors in $\ff{4}$. Our main result for $4$-wise intersection patterns is as follows:

\begin{thm}\label{thm:aplha0110}
As $n \to \infty$,
\[ f_{(0,1,1,0)}(n) \sim \sqrt{2n} \hspace{5mm} \text{and} \hspace{5mm} f_{(0,0,0,1)}(n) \sim \sqrt{2n}.  \]
\end{thm}

\cref{thm:aplha0110}, together with a few short arguments in \cref{sec:4wise} will establish the following table: 
{\hypersetup{linkcolor = blue}
\begin{table}[h!]
\centering
{\renewcommand{\arraystretch}{1.5}
\begin{tabular}{|c|c|c||c|c|c|}
\hline
$\alpha$ & section & $f_{\alpha} (n)$ & $\alpha$ & section & $f_{\alpha} (n)$ \\
\hline
\hline
$(0,0,0,0)$ & \cref{subsec:0000} & $2^{\lfloor n/2 \rfloor}$ 
& $(1,0,0,0)$ & \cref{subsec:1000} & $n$ \\ 
\hline 
$(1,0,1,0)$ & \cref{subsec:1010} & $ \begin{cases}
            n  & \text{$n$ even} \\
            n-1  & \text{$n$ odd}
        \end{cases} $ 
& $(0,0,1,0)$ & \cref{subsec:0010}  &  $  \begin{cases}
        \lfloor n/2 \rfloor + 1 & \text{
        $n \equiv 3 \pmod{4}$} \\
        \lfloor n/2 \rfloor & 
        \text{otherwise}
    \end{cases}$ \\ 
\hline
$(0,1,0,0)$ & \cref{subsec:0100}  & $\sim  \sqrt{2n}$ & 
$(0,0,1,1)$ & \cref{subsec:0011} & $ \sim \sqrt{2n}$ 
 \\ 
\hline
$(0,1,1,0)$ & \cref{subsec:0110} & $\sim  \sqrt{2n}$ 
& $(0,0,0,1)$ & \cref{subsec:0001} & $ \sim \sqrt{2n}$ \\ 
\hline
\end{tabular}
}
\caption{Our $4$-wise results for $n\geq 7$}
\end{table}
}

For both $\ff{3}$ and $\ff{4}$, there are only eight such vectors $\alpha$ for which there exists a linear size construction $\F \subset 2^{[n]}$ that satisfies $\alpha$-intersection pattern modulo $2$ (i.e. $f_\alpha(n) = \Omega(n)$). These vectors $\alpha$ correspond to $\alpha=\vb{0}$ (Strong Eventown), $\alpha \in \{(1,0,0),(1,0,0,0)\}$ (Strong Oddtown), $\alpha=\{(1,0,1), (1,0,1,0)\} $ (Alternating Oddtown) and $\alpha = \{ (0,0,1), (0,0,1,0)\}$ (Delayed Alternating Oddtown) and their corresponding duals. We are able to show that for $k \geq 5$, there are also only eight such vectors with $f_\alpha(n)=\Omega(n)$ in the following strong form:

\begin{thm}\label{thm:genk}
Fix $k \geq 3$. Let $\alpha \in \ff{k}$ be so that neither $\alpha$, nor  $\alpha+\vb{1}$ are in the set \\ $\{ \vb{0}, (1,0,\ldots,0), (1,0,1,0,\ldots), (0,0,1,0,1,0,\ldots) \}$. Then $f_\alpha(n) = O(\sqrt{n})$.     
\end{thm}

As with the classical oddtown and eventown problems, it is natural to ask about modulo $p$ variants. In this paper, we also consider intersection patterns modulo $3$ for $\alpha \in \ff{3}$. To generalize our results to the modulo $3$ setting, we adopt the convention that we restrict our intersections to either be zero modulo $3$ or nonzero modulo $3$. We will write $\alpha_i = \star$ to indicate that we are requiring all distinct $i$-wise intersections to be nonzero modulo $3$ (i.e. the sizes are in $\{1,2\} \pmod{3}$). Abusing notation, for a vector $\alpha = (\alpha_1, \ldots, \alpha_k) \in \{ 0, \star \}^k$, we say a collection of subsets $\F$ satisfies {\it $\alpha$-intersection pattern modulo $3$} if all $i$-wise intersections consisting of $i$ distinct sets from $\F$ have size $\alpha_i \pmod{3}$. We let $g_\alpha(n)$  denote the maximum size family $\F \subset 2^{[n]}$ where $\F$ satisfies $\alpha$-intersection pattern modulo $3$.

In this language, the classical modulo $3$ oddtown problem yields that $g_{(\star,0)}(n) = n$. In stark contrast to the modulo $2$ setting, there is no duality. In fact, $g_{(0,\star)}(n) = \Theta(n^2)$, with a lower bound coming from the star $\{ A \in \binom{[n]}{3} : 1 \in A\}$ and the upper bound from the Deza-Frankl-Singhi Theorem \cite{DFS} (see also \cite[Theorem 5.15]{BF}) since satisfying a $(0,\star)$-intersection pattern is equivalent to being a $(3,\{1,2\})$-intersecting family. In \cref{sec:mod3problems}, we will show the 
results in \cref{table:3wise_mod3}.
{\hypersetup{linkcolor = blue}
\begin{table}[h!]
\centering
{\renewcommand{\arraystretch}{1.25}
\begin{tabular}{|c|c|c||c|c|c|}
\hline
$\alpha$ & section/reference & $g_{\alpha} (n)$ & 
$\alpha$ & section/reference & $g_{\alpha} (n)$ \\
\hline
\hline
$(\star,0,0)$ & \cref{subsec:star00} &  $n$ & 
$(0,\star,\star)$ & \cref{subsec:0starstar}  &  $\sim n^2/2$ \\
\hline
$(\star,0,\star)$ & \cref{subsec:star0star} & $\sim n$ & 
$(0,\star,0)$ & \cref{subsec:0star0} &  $\Theta(n) $ \\
\hline

$(\star,\star,0)$ & \cref{subsec:starstar0} & $\sim n$ & 
$(0,0,\star)$ & \cref{subsec:00star}  & $\Omega(n)$; $O(n^2)$ \\ 

\hline
$(0,0,0)$ &  \cite{FO,SV} &  $\Omega(c^n)$ &
$(\star,\star,\star)$ & \cite{FO,SV}   & $\Omega(c^n)$\\ 
\hline

\end{tabular}
}
\caption{The $\Mod{3}$ results}
\label{table:3wise_mod3}
\end{table}
}

In the case where $\alpha=(0,0)$, a natural construction is to partition the ground set into sets of size three and take all possible unions, yielding $2^{\lfloor n/3 \rfloor}$ subsets. Frankl and Odlyzko \cite{FO} give a construction utilizing Hadamard matrices which in our particular setting shows that $g_{(0,0)}(n) = \Omega(c^n)$ for $c=24^{1/12} \approx 1.3$ whereas $2^{1/3} \approx 1.26$. By adding an auxiliary element to each set, one also gets a construction of a family that satisfies $(\star,\star)$-intersection pattern modulo $3$ and shows $g_{(\star,\star)}(n) = \Omega(c^n)$ for $c=24^{1/12} \approx 1.3$. For $(0,0,0) \in \{0,\star\}^{3}$, Sudakov and Vieira \cite[Lemma 16]{SV} prove a connection between families satisfying $ \vb{0} \in \{0,\star\}^{k-1}$ intersection pattern modulo $p$ and families satisfying $ \vb{0} \in \{0,\star\}^{k}$ intersection pattern modulo $p$. Hence, the Frankl and Odlyzko \cite{FO} construction and Sudakov and Vieira \cite{SV} lemma collectively fill in the fourth row of the above table. We make no attempt to improve the constant of $c$ in this paper. However, it is also worth noting that very recent work of Gishboliner, Tomon and Sudakov \cite{GishbolinerTomonSudakov} implies that for all $\ell \geq 1$, there exists a $k=k(\ell)$ so that the canonical construction of $2^{\lfloor n/ \ell \rfloor}$ subsets is essentially (i.e. off by a constant) extremal amongst families which satisfy $\vb{0} \in \{0,\star\}^{k}$ intersection pattern modulo $\ell$. 

The other open case is $\alpha= (0,0,\star)$ for which we will show $g_{(0,0,\star)}(n)= \Omega(n)$ and $g_{(0,0,\star)}(n)= O(n^2)$. We leave determining the value of $g_{(0,0,\star)}(n)$ as an interesting open problem: 

\begin{problem}
Determine the value of $g_{(0,0,\star)}(n)$.      
\end{problem}
  
{\bf Structure:} 
In \cref{sec:prelim}, we will prove a few lemmas towards determining $f_n(\alpha)$ when $\alpha \in \ff{k}$ for $k=3,4$ and also general values of $k$. We will prove \cref{thm:alpha001} and the remaining values from Table 1 in \cref{sec:3wise}. In \cref{sec:4wise}, we will prove \cref{thm:aplha0110} and the remaining values from Table 2 as well as prove \cref{thm:genk}. We will then discuss Table 3 in \cref{sec:mod3problems}.

\section{Preliminaries}\label{sec:prelim}
In this section, we will prove three lemmas that we will utilize throughout. The following lemma establishes a connection between $f_\alpha(n)$ and 
$f_\beta(n)$ when $\alpha$ contains $\beta$ as a consecutive 
substring.  
\begin{lemma}[Trace Lemma]\label{lemma:trace}
    Let $\alpha = (\alpha_1,\dots, \alpha_k) \in \mathbb{F}_2^k$ and 
    let $\beta = (\alpha_{t+1},\dots, \alpha_{t+r}) \in \mathbb{F}_2^r$ where 
    $t+r \leq k$, $f_\alpha(n) > k$, and $\alpha_{t+1} \neq \alpha_{t+2}$. 
    Then, if $\F$ satisfies $\alpha$-intersection pattern modulo $2$ and $F_1, \ldots, F_t \in \F$,   
    \begin{equation*}
        |\F| \leq f_\beta(|F_1 \cap \cdots \cap F_t|) + t    \end{equation*}
\end{lemma}
\begin{proof}
    Let $\F$ and
    $F_1,\dots, F_{t} \in \F$ be as in the lemma statement and set
    $T = F_1 \cap \cdots \cap F_{t}$. Then 
    \[
        \F_{T} := \{F \cap T \ : \ F \in 
        \F \setminus \{F_1,\dots, F_{t}\} \}
    \]
    is a $\beta$-intersecting family. Note also that if 
    $A,B \in \F \setminus \{F_1,\dots, F_{t}\}$ with 
    $A\neq B$
    and 
    $A\cap T = B\cap T$, then 
    $A \cap B \cap T = (A\cap T) \cap (B\cap T) = A \cap T$. But this 
    is impossible as it implies $\alpha_{t+1} = 
    \alpha_{t+2}$, a contradiction. Since $|\F|= |\F_T| + t$ and $\F_T$ satisfies $\beta$-intersection pattern modulo $2$ with ground set $|T|$, the desired bound holds.  
\end{proof}

We call the intersection pattern $\alpha + \mathbf{1}$ (where $\mathbf{1}$ is the all $1$'s vector) \textit{the dual} 
of the $\alpha$-intersection pattern. The next lemma establishes 
that for $\alpha \in \mathbb{F}_2^k \setminus \{\mathbf{0},\mathbf{1}\}$, the functions $f_{\alpha}(n)$ and $f_{\alpha + \mathbf{1}}(n)$ are very similar. 
\begin{lemma}[Dual Lemma]\label{The Dual Lemma}
    Let $\alpha \in \mathbb{F}_2^k \setminus \{\mathbf{0},\mathbf{1}\}$. Then $f_{\alpha}(n-1) \leq f_{\alpha + \mathbf{1}}(n) \leq f_{\alpha}(n+1)$. Thus, if $f_{\alpha +\mathbf{1}}(n+1)/f_{\alpha+\mathbf{1}}(n) \to 1$ as $n \to \infty$, $f_{\alpha}(n) \sim f_{\alpha + \mathbf{1}}(n)$. 
\end{lemma}
\begin{proof}
    Let $\F \subseteq 2^{[n]}$ be an $\alpha$-intersecting family. 
    Consider the family $\F' 
    \subseteq 2^{[n+1]}$ where  
    \[
        \F' := \{F \cup \{n+1\} \ : \ F \in \F\}. 
    \]
    Note that the size of each of the $m$-wise intersection also increases by exactly one for each 
    $1\leq m \leq k$. Thus $\F'$ satisfies 
    $(\alpha + \mathbf{1})$-intersection pattern modulo $2$ and 
    $|\F| = |\F'| \leq f_{\alpha+\mathbf{1}}(n+1)$. By a similar argument, it follows that  
    $f_{\alpha + \mathbf{1}}(n-1) \leq f_{\alpha}(n) \leq f_{\alpha + \mathbf{1}}(n+1)$ and 
    thus 
    \begin{equation}\label{dual_lemma_displayed}
        \frac{f_{\alpha + \mathbf{1}}(n-1)}{f_{\alpha + \mathbf{1}}(n)} \leq \frac{f_{\alpha}(n)}{f_{\alpha + \mathbf{1}}(n)} \leq \frac{f_{\alpha + \mathbf{1}}(n+1)}{f_{\alpha + \mathbf{1}}(n)}. 
    \end{equation}

    By \eqref{dual_lemma_displayed} and using the condition $f_{\alpha +\mathbf{1}}(n+1)/f_{\alpha+\mathbf{1}}(n) \to 1$ as $n \to \infty$ twice, the squeeze theorem then implies that $f_{\alpha}(n) \sim f_{\alpha + \mathbf{1}}(n)$. \qedhere 
    
\end{proof}

It is worth noting that the condition $f_{\alpha +\mathbf{1}}(n+1)/f_{\alpha+\mathbf{1}}(n) \to 1$ as $n \to \infty$ holds for all $\alpha \in \mathbb{F}_2^k \setminus \{\mathbf{0},\mathbf{1}\}$ with $k \leq 4$ and is false for $\alpha \in  \{\mathbf{0},\mathbf{1}\}$. We leave the problem of determining if this is true for all such $\alpha$ or finding a counterexample as an interesting open problem to the reader.

The below lemma is inspired by a lemma of Frankl and Odlyzko in 
\cite[Lemma 2]{FO} and allows one to obtain a lower bound 
of $f_{\gamma}(n)$ in terms of $f_{\alpha}(\cdot)$ and $f_{\beta}(\cdot)$ where 
$\alpha + \beta = \gamma$. 
\begin{lemma}[Partition Sum Lemma]\label{lemma:PartitionSumLemma}
    Let $\alpha = (\alpha_1,\dots, \alpha_k), 
    \beta = (\beta_1,\dots, \beta_k)$ be vectors in $\mathbb{F}_2^k$ and let 
    $\gamma = \alpha + \beta$.
    Then for any $r \in [n-1]$ where 
    $f_{\alpha}(r), f_{\beta}(n-r) > k$ we have 
    \[
        f_\gamma(n) \geq \min\{f_{\alpha}(r),f_{\beta}(n-r)\}.
    \]
\end{lemma} 
\begin{proof}
Let $\ell := f_\alpha(r)$ and $m := f_\beta(n-r)$.
Let 
$\A = \{A_1,\dots, A_{\ell}\}$ be a family with ground set $[r]$ satisfying $\alpha$-intersection pattern 
modulo $2$ and let $\B = \{B_1,\dots, B_{m}\}$ be a family with ground set $\{r+1,\dots, n\}$ 
satisfying $\beta$-intersection pattern modulo $2$. Consider 
the family 
\[
    \F := \{A_i \cup B_i \ : \ 1 \leq i \leq \min\{\ell,m\}\}.
\]
Since the ground sets of $\A$ and $\B$ are disjoint, $|\F| = \min\{f_{\alpha}(r), f_{\beta}(n-r)\}$. Thus we need  
only show that $\F$ satisfies $\gamma$-intersection pattern modulo $2$. 
Using again that 
the ground sets of $\A$ and $\B$ are disjoint,  
for any $t$ satisfying $1 \leq t \leq k$, 
\[
    \left|\bigcap_{i=1}^t (A_i \cup B_i) \right| = \left|\bigcap_{i=1}^t A_i \right| + \left|\bigcap_{i=1}^t B_i \right|  \equiv \alpha_t + \beta_t = \gamma_t. \qedhere 
\]
\end{proof}

\section{\texorpdfstring{$3$}{}-wise modulo \texorpdfstring{$2$}{} Patterns}\label{sec:3wise}
\subsection{The \texorpdfstring{$(1,0,0)$}{} Problem}\label{subsec:100}
Any family $\F$ that satisfies $(1,0,0)$-intersection pattern modulo $2$ necessarily satisfies $(1,0)$-intersection pattern modulo $2$ (oddtown rules). This, together with the fact that the set of all singletons follows $(1,0,0)$-intersection pattern modulo $2$, yields 
$f_{(1,0,0)}(n) = n$.
\subsection{The \texorpdfstring{$(0,1,1)$}{} Problem}\label{subsec:011}
Here we show that 
$f_{(0,1,1)}(n) = n-1$. 
Note that 
the $2$-uniform star $\{(i,1) \ i \in \{2,3,\dots, n\}\}$ satisfies  
$(0,1,1)$-intersection pattern modulo $2$ and hence witnesses the 
lower bound $n-1$. To show a matching upper bound, let $\F = \{F_1,\dots, F_m\}$ satisfy  
$(0,1,1)$-intersection pattern modulo $2$ and  
let $F_i \in \F$. Consider the family 
\[
    \A = \{F \cap F_i^C \ : \ F \in \F \setminus \{F_i\}\}.
\]
Note $\A$ satisfies $(1,0)$-intersection pattern modulo $2$ since  
$|F \cap F_i^C| = |F| - |F \cap F_i| \equiv 1 \pmod{2}$, and 
\[
    |(F \cap F_i^C) \cap (F' \cap F_i^C)| = |F_i^C \cap F \cap F'| = 
    |F \cap F'| - |F \cap F' \cap F_i| \equiv 0 \pmod{2}.
\]
As $\F$ cannot contain any singletons, 
$|F_i|\geq 2$.
Hence 
$|\F| - 1 \leq |F_i^C| \leq n-2$,
as desired.
\subsection{The \texorpdfstring{$(0,0,0)$}{} Problem}\label{subsec:000}
Any family $\F$ that satisfies 
$(0,0,0)$-intersection pattern modulo $2$
necessarily satisfies 
$(0,0)$-intersection pattern modulo $2$. Hence 
$|\F| \leq 2^{\lfloor n/2 \rfloor}$.
On the other hand the canonical extremal  
construction for families satisfying $(0,0)$-intersection pattern 
also satisfies the $(0,0,0)$-intersection pattern.
Thus $f_{(0,0,0)}(n) = 2^{\lfloor n/2 \rfloor}$.
\subsection{The \texorpdfstring{$(1,1,1)$}{} Problem}\label{subsec:111}
We show that $f_{(1,1,1)}(n) = 2^{\lfloor (n-1)/2\rfloor}$. Since any family satisfying $(1,1,1)$-intersection pattern modulo $2$ also satisfies $(1,1)$-intersection 
pattern modulo $2$, we have 
from the dual eventown theorem (see \autoref{dual-eventown-canoncial} in the appendix) 
\[
    f_{(1,1,1)}(n) \leq f_{(1,1)}(n) = \begin{cases}
        2^{\frac{n-1}{2}} & \text{$n$ odd} \\
        2^{\frac{n-2}{2}} & \text{$n$ even} \; .
    \end{cases}
\]
These upper bounds are best possible as the canonical extremal family satisfying 
$(1,1)$-intersection pattern modulo $2$
of \autoref{dual-eventown-canoncial} also satisfies $(1,1,1)$-intersection pattern modulo $2$. 
\subsection{The \texorpdfstring{$(1,0,1)$}{} Problem}\label{subsec:101} 
We will show that 
\[
    f_{(1,0,1)}(n) = 
    \begin{cases} 
        n & \text{$n$ even} \\
        n-1 & \text{$n$ odd} \; .
    \end{cases}
\]
\textbf{Case 1: } Let $n$ be even.  
Since any family satisfying $(1,0,1)$-intersecting pattern modulo $2$ satisfies $(1,0)$-intersection pattern 
modulo $2$, 
it follows that $f_{(1,0,1)}(n) \leq n$. Note also that since $n$ is even, the family 
$\binom{[n]}{n-1}$ satisfies $(1,0,1)$-intersection pattern modulo $2$ and has size $n$. 
\par 
\textbf{Case 2: } For $n$ odd, 
    we look at the family 
    $\binom{[n-1]}{n-2}$ to get a family of size $n-1$. 
    To prove the upper bound,  
    let $\F$ satisfy $(1,0,1)$-intersection pattern modulo $2$.  Since $n$ is odd, we cannot 
    have sets of size $n-1$. Moreover, if $[n] \in \F$ then $\F$ cannot have more than one element 
    (since $|F\cap [n]| = |F|$ is odd for any $F \in \F$). Thus, all sets in $\F$ have size at most  
    $n-2$. We are then done by the \nameref{lemma:trace} together with the dual oddtown theorem. 
    \par 
    As an aside one can show that the above example is the unique extremal example by proceeding as follows: 
    If there exists a set of size strictly smaller than $n-2$, then the \nameref{lemma:trace} together with the dual oddtown theorem imply that $|\F| \leq n-3$. Thus, we must have $|F|= n-2$ for all $F \in \F$ and for all $F_1, F_2 \in \F$, we have $|F_1 \cap F_2| = n-3$ as $n-4$ is odd. Consider now 
    $\F^C := \{F^C \ : \ F \in \F\} \subseteq \binom{[n]}{2}$.
    Note that $|F_1^C \cap F_2^C| = |(F_1\cup F_2)^C| = 1$ for all $F_1, F_2 \in \F$. 
    Thus, $\F^C$ is an {\it intersecting} family of $2$-element sets. By the Erd\H{o}s-Ko-Rado Theorem \cite{EKR}, 
    \[
        |\F| = |\F^C| \leq \binom{n-1}{2-1} = n-1,
    \]
    with equality if and only if 
    $\F^C$ is a star for $n\geq 5$.

\subsection{The \texorpdfstring{$(0,1,0)$}{} Problem}\label{subsec:010}
Note that a family satisfying $(0,1,0)$-intersection pattern modulo $2$ necessarily satisfies  
$(0,1)$-intersection pattern modulo $2$, giving  
\[
    f_{(0,1,0)}(n) \leq f_{(0,1)}(n) = 
    \begin{cases} 
        n-1 & \text{$n$ is even} \\
        n & \text{$n$ is odd} 
    \end{cases}
\]
When $n$ is odd this upper bound is witnessed 
by the family $\binom{[n]}{n-1}$. When 
$n$ is even the upper bound is witnessed by the 
family $\binom{[n-1]}{n-2}$.

\subsection{The \texorpdfstring{$(1,1,0)$}{} Problem}\label{subsec:110}
We will show that 
\[
    f_{(1,1,0)}(n) = 
    \begin{cases}
        \lfloor n/2 \rfloor + 1 & \text{$n \equiv 2,3 \pmod{4}$} \\
        \lfloor n/2 \rfloor & \text{$n \equiv 0,1 \pmod{4}$}
    \end{cases}
\]
We start with the below construction.
\begin{const}\label{construction:110_family}
    \textbf{Case  1:} $n \equiv 2 \pmod{4}$. Let $n = 4k+2$ and  
    define $\A = \{A_0,A_1,\dots, A_{2k+1}\}$ where 
\[ 
    A_0 = [2k+2, 4k+2] 
    \quad \text{and} \quad 
    A_i = ([2k+1] \setminus \{i\}) \cup \{2k+1 + i\} \quad \text{for $i \in [2k+1]$},
\]
which are easily shown to satisfy $(1,1,0)$ rules. 
\begin{figure}[h!]
    \centering
    \begin{tikzpicture}[thick,
  fsnode/.style={circle, fill=black, inner sep = 2pt},
  ssnode/.style={circle, draw, inner sep = 2pt},
    shorten >= -5pt,shorten <= -5pt
]
\draw[very thick] (-5,0) -- (0,0) ; 
\draw[very thick, dotted] (0,0) -- (5,0);
\node[ssnode, label={[xshift=0em, yshift=0em] \small$i$}] at (-3,0) (a) {};
\node[fsnode, label={[xshift=0em, yshift=0] \small$2k+1+i$}] at (2,0) (a) {};
\end{tikzpicture}
\caption{The sets in \protect\autoref{construction:110_family}}
\end{figure}

\par 
\textbf{Case 2: } $n\not\equiv 2 \Mod{4}$. 
We consider the same construction above on 
$4k+2$ vertices with isolated vertices. 
More precisely  
\begin{align*}
\begin{array}{c | c | c}
\text{ground set} & \text{construction} & \text{size of family}
\\
\hline\hline
n= 4k & \text{construction on $[4k-2]$ with $2$ isolated vertices} & 2k = n/2  \\
n = 4k+1 & \text{construction on $[4k-2]$ with $3$ isolated vertices} & 2k = \lfloor n/2 \rfloor
\\
n = 4k+2 & \text{construction on $[4k+2]$ with $0$ isolated vertices} &  2k+2 = n/2 + 1 \\
n = 4k+3 & \text{construction on $[4k+2]$ with 
$1$ isolated vertex} & 2k + 2 = \lfloor n/2 \rfloor + 1
\end{array}
\end{align*}

One can readily check that $\A$ satisfies $(1,1,0)$-intersection pattern 
modulo $2$ in each of the above cases. 
\end{const}

This construction is best possible as we will now show. 
The below lemma\footnote{a Corollary of the Trace Lemma} says that any family satisfying $(1,1,0)$-intersection 
pattern modulo $2$ must be small if it contains small 
sets. 
\begin{lemma}\label{lemma:110sets}
Let $\F  = \{F_1,\dots, F_m\} \subset 2^{[n]}$ satisfy $(1,1,0)$-intersection pattern modulo $2$  and consider any $F_i \in \F$. Then $|\F|-1 \leq |F_i|$. 
\end{lemma}
\begin{proof}
Consider the family $\F_i = \{ F_j \cap F_i: j \neq i\}$ and note that this family satisfies oddtown rules on ground set $F_i$. As such, the stated bound holds. 
\end{proof}
The next lemma gives us that any family satisfying 
$(1,1,0)$-intersection pattern modulo $2$ must 
be small if it contains large sets.
\begin{lemma}\label{lemma:110complement}
Let $\F = \{F_1,\dots, F_m\} \subset 2^{[n]}$ satisfy $(1,1,0)$-intersection pattern modulo $2$. Then for any $F_i \in \F$ 
\begin{itemize}
    \item[(i)] $|\F|-1 \leq |F_i^C|$, and
    \item[(ii)] $|\F| -1 \leq |F_i^C| - 1$ when $n$ is 
    odd.
\end{itemize} 
\end{lemma}
\begin{proof}
Consider the family $\F_i = \{ F_j \cap F_i^C: j \neq i\}$. 
We will show that $\F_i$ satisfies $(0,1)$-intersection 
pattern modulo $2$. 
First note that $|F_j \cap F_i^C|$ is even as $|F_j \cap F_i^C| 
= |F_j| - |F_j \cap F_i|$ (since 
$F_j$ is the disjoint union of 
$F_j \cap F_i$ and $F_j\cap F_i^C$) 
and $|F_j|$ and $|F_j \cap F_i|$ are both odd. Further 
(noting again that 
$F_j \cap F_k$ is the disjoint union of 
$F_j \cap F_k \cap F_i^C$ and $F_j \cap F_k \cap F_i$) we have that 
\[ 
|F_j \cap F_i^C \cap F_k \cap F_i^C|= |F_j \cap F_k \cap F_i^C|= |F_j \cap F_k|- |F_j \cap F_k \cap F_i| 
\] 
is odd as $|F_j \cap F_k|$ is odd and $|F_j \cap F_k \cap F_i|$ is even. Thus 
$\F_i$ satisfies $(0,1)$-intersection pattern modulo $2$ and the desired bound holds by dual oddtown rules. 
We obtain (ii) by recalling that 
$f_{(0,1)}(n) = n-1$ for $n$ even and $|F_i^C| = n - |F_i|$ is even when $n$ is odd. 
\end{proof}
Using \autoref{lemma:110sets} and \autoref{lemma:110complement}, for any 
$F_i \in \F$ we get  
\begin{equation}\label{eq:110upperbound}
 (|\F|-1)+ (|\F|-1) \leq |F_i| + |F_i^C| = n   
 \implies |\F| \leq \lfloor n/2 \rfloor + 1 
\end{equation}
We will now sharpen the above bound depending on 
the residue of $n$ modulo $4$. Throughout the rest of this section, let $\F = \{F_1,\dots, F_m\} \subseteq 2^{[n]}$ satisfy  
$(1,1,0)$-intersection pattern modulo $2$
\par 
\textbf{(Case 1: $n\equiv 2,3 \pmod{4}$).}  Let $F_i \in \F$.
We first note that for 
$n$ even, one of $F_i$ or $F_i^C$ has size 
at most $n/2$ and for $n$ odd 
either $F_i$ or $F_i^C$ has size at most $\frac{n-1}{2}$. 
Thus, combining \cref{lemma:110sets} and \cref{lemma:110complement} we have 
\[
    |\F| - 1 \leq \min\{|F_i|, |F_i^C|\} \leq 
    \begin{cases}
        \frac{n}{2} & \text{$n$ even} \\
        \frac{n-1}{2} & \text{$n$ odd}
    \end{cases}
\]
Thus, \autoref{construction:110_family} is best possible 
when $n \equiv 2,3  \pmod{4}$. 
\par 
\textbf{(Case 2: $n\equiv 0 \pmod{4}$).} When $n \equiv 0 \pmod{4}$, we note that for $F_i \in \F$, both $F_i$ and 
$F_i^C$ must have odd size and $\frac{n}{2}$ must be even. Thus, 
$|\F| - 1 \leq \min\{|F_i|,|F_i^C|\} \leq \frac{n}{2} - 1$, giving us the desired bound. 
\par 
\textbf{(Case 3: $n\equiv 1 \pmod{4}$).} When $n \equiv 1 \pmod{4}$, for $F_i \in \F$, 
note that 
$|F_i^C|$ has even size. Thus by 
part (ii) of \autoref{lemma:110complement} 
we get the slightly better bound 
of $|\F| -1 \leq |F_i^C| - 1$. Letting 
$n = 4k+1$ we have that if there is a set $F_i \in \F$ such that 
$|F_i| \leq 2k-1$ then 
\[
    |\F| - 1 \leq 2k - 1 \implies |\F| \leq 2k
\]
as desired. Further, if there is a set $F_j \in \F$ such that 
$|F_j| \geq 2k+1$ then we have 
\begin{align*}
    |\F| - 1 &\leq |F_j^C| - 1  \\
    &= n - |F_j| - 1 \\
    &\leq 4k+1 - (2k+1) - 1 \\
    &= 2k - 1. 
\end{align*}
This gives us $|\F| \leq 2k$ as desired. The only other possibility is 
that every $F \in \F$ has size $2k$. But this is impossible as $\F$ satisfies 
$(1,1,0)$-intersection pattern modulo $2$ and hence has odd-sized sets.

\subsection{The \texorpdfstring{$(0,0,1)$}{} Problem}\label{subsec:001}
We show that 
\[
    f_{(0,0,1)}(n) = 
    \begin{cases}
        \lfloor n/2 \rfloor + 1 & \text{
        $n \equiv 3 \pmod{4}$} \\
        \lfloor n/2 \rfloor & \text{$n\equiv 0,1,2 \pmod{4}$} 
    \end{cases}
\]
For the lower bound, we make a slight adjustment to 
\autoref{construction:110_family}.
\begin{const}\label{construction:001_family}
\textbf{Case  1:} $n \equiv 3 \pmod{4}$. Let 
$n = 4k+3$ and 
define $\A = \{A_0,A_1,\dots, A_{2k+1}\}$ 
where 
\[ 
    A_0 = [2k+2, 4k+3] 
    \quad \text{and} \quad 
    A_i = ([2k+1] \setminus \{i\}) \cup \{2k+1 + i, 4k+3\} \quad \text{for $i \in [2k+1]$}.
\]

\begin{figure}[h!]
    \centering
    \begin{tikzpicture}[thick,
  fsnode/.style={circle, fill=black, inner sep = 2pt},
  ssnode/.style={circle, fill=red, inner sep = 2pt},
  esnode/.style={circle, draw, inner sep = 2pt},
    shorten >= -5pt,shorten <= -5pt
]
\draw[very thick] (-5,0) -- (0,0) ; 
\draw[very thick, dotted] (0,0) -- (5,0);
\node[esnode, label={[xshift=0em, yshift=0em] \small$i$}] at (-3,0) (a) {};
\node[fsnode, label={[xshift=0em, yshift=0] \small$2k+1+i$}] at (2,0) (a) {};
\node[fsnode, label={[xshift = 0em, yshift=-2em] \small$4k+3$}] at (5.15,0) (a) {};
\end{tikzpicture}
\caption{The sets in \protect\autoref{construction:001_family}}
\end{figure}
\par 
\textbf{Case 2: } $n\not\equiv 3 \pmod{4}$. 
We add isolated vertices in a similar 
manner as \cref{construction:001_family}:
\begin{align*}
\begin{array}{c | c | c}
\text{ground set} & \text{construction} & \text{size of family}
\\
\hline\hline
n= 4k & \text{construction on $[4k-1]$ with $1$ isolated vertex} & 2k = n/2  \\
n = 4k+1 & \text{construction on $[4k-1]$ with $2$ isolated vertices} & 2k = \lfloor n/2 \rfloor
\\
n = 4k+2 & \text{construction on $[4k-1]$ with $3$ isolated vertices} &  
2k =  n/2 \\
n = 4k+3 & \text{construction on $[4k+3]$ with 
$0$ isolated vertices} & 2k + 2 = \lfloor n/2 \rfloor + 1
\end{array}
\end{align*}

One can readily check that this family satisfies $(0,0,1)$-intersection pattern modulo $2$ in 
each of the above cases. 
\end{const}
For the upper bound, we first obtain ``duals" of \autoref{lemma:110sets} and \autoref{lemma:110complement} below. 
\begin{lemma}\label{001_trace_lemma}
    Let $\F = \{F_1,\dots, F_m\} \subset 2^{[n]}$ satisfy $(0,0,1)$-intersection pattern modulo $2$.  Then, for any $F_i \in \F$,  
    \begin{itemize}[noitemsep]
        \item[(i)] $|\F| - 1  \leq |F_i| - 1$,  
        \item[(ii)] $|\F| - 1 \leq |F_i^C|$ (when $n$ is odd), and 
        \item[(ii-e)] $|\F| - 1 \leq |F_i^C| - 1$ (when $n$ is even).
    \end{itemize}
\end{lemma}
\begin{proof}
The family $\F_i := \{F \cap F_i \ : \ F\in \F, F \neq F_i\}$ satisfies $(0,1)$-intersection pattern modulo $2$ 
and hence has size 
at most $|F_i| - 1$ (since $|F_i|$ is even). 
The family $\F_i' := \{F \cap F_i^C \ : \ F \in \F, F \neq F_i\}$ satisfies $(0,1)$-intersection pattern 
modulo $2$ with 
ground set $|F_i^C|$. (This can be seen by using exactly the same 
reasoning as in \cref{lemma:110complement}.)
When $n$ is even, we obtain the slightly better bound $|\F| - 1 \leq |F_i^C| - 1$ since 
$\F_i$ satisfies $(0,1)$-intersection pattern modulo $2$ with ground 
set of size $|F_i^C| = n - |F_i|$, which is even. 
\end{proof}
We now match the bounds from \cref{construction:001_family} in a similar way as 
was done in \cref{subsec:110}. Again, throughout the 
remainder of this section 
we 
let $\F = \{F_1,\dots, F_m\} \subset 2^{[n]}$ satisfy $(0,0,1)$-intersection pattern modulo $2$.
\par 
\textbf{(Case 1: $n\equiv 0,2,3 \pmod{4}$).}
    For $n$ even, (i) and (ii-e) of \autoref{001_trace_lemma} gives that 
    for any $F_i \in \F$, 
    \[
        (|\F| - 1) + (|\F| - 1) \leq (|F_i| - 1) + (|F_i^C| - 1) = n - 2 \implies 
        |\F| \leq \frac{n}{2}.
    \]
    Similarly, for $n$ odd,  
    \[
        (|\F| - 1) + (|\F| - 1) \leq (|F_i| - 1) + |F_i^C| = n - 1 \implies 
        |\F| \leq \frac{n+1}{2} = \left\lfloor \frac{n}{2} \right\rfloor + 1.
    \]
    \par 
\textbf{(Case 2: $n\equiv 1 \pmod{4}$).}
By \cref{001_trace_lemma},
if $F_i \in \F$ with $|F_i| \leq 2k - 2$, then  
\[
    |\F| -1 \leq |F_i| - 1 \leq 2k-3 \implies |\F| \leq 2k-2 \leq \left\lfloor \frac{n}{2} \right\rfloor.
\]
Similarly, if $F_i \in \F$ with $|F_i| \geq 2k+2$, then  
\[
    |\F| - 1 \leq |F_i^C| = n - |F_i|  \leq 4k+1 - (2k+2) = 2k-1 \implies 
    |\F| \leq 2k = \left\lfloor \frac{n}{2} \right\rfloor.
\]
The only other possibility is that every set in $\F$ is of size $2k$. 
However, in this case, we may apply the dual oddtown result for 
even sized ground set 
together with \cref{001_trace_lemma} 
to obtain 
\[
    (|\F| - 1) + (|\F| - 1) \leq (|F_i| - 1) + (|F_i^C| - 1) \implies 
    2|\F| \leq n \implies |\F| \leq \left\lfloor \frac{n}{2} \right\rfloor.
\]

\section{\texorpdfstring{$4$}{}-wise modulo \texorpdfstring{$2$}{} Patterns}\label{sec:4wise}
\subsection{The \texorpdfstring{$(0,0,0,0)$}{} Problem}\label{subsec:0000} 
Any family satisfying $(0,0,0,0)$-intersection pattern modulo $2$ also satisfies 
$(0,0)$-intersection pattern modulo $2$. Hence, 
$f_{(0,0,0,0)}(n) \leq f_{(0,0)}(n) = 2^{\lfloor n/2 \rfloor}$. 
Moreover, the canonical extremal eventown construction satisfies 
$(0,0,0,0)$-intersection pattern modulo $2$. Thus $f_{(0,0,0,0)}(n) = 2^{\lfloor n/2 \rfloor}$.

\subsection{The \texorpdfstring{$(1,0,0,0)$}{} Problem}\label{subsec:1000}
We immediately obtain 
$f_{(1,0,0,0)}(n) \leq f_{(1,0)}(n) = n$.
Similarly, the singletons satisfy $(1,0,0,0)$-intersection pattern modulo $2$ giving us $f_{(1,0,0,0)}(n) = n$.

\subsection{The \texorpdfstring{$(1,0,1,0)$}{} Problem}\label{subsec:1010}
We show that 
\[
    f_{(1,0,1,0)}(n) = 
    \begin{cases}
        n & \text{if $n$ is even} \\
        n-1 & \text{if $n$ is odd}
    \end{cases}
    \ 
    .
\]
We first note that 
$f_{(1,0,1,0)}(n) \leq f_{(1,0)}(n) = n$. 
For $n$ even, the family $\binom{[n]}{n-1}$ is a  
family of size $n$ satisfying 
$(1,0,1,0)$-intersection pattern modulo $2$. When $n$ is 
odd, we may look at $\binom{[n-1]}{n-2}$ to obtain a lower 
bound of $n-1$, and this is best possible as can be seen via the \nameref{lemma:trace}. 
Indeed, let $\F$ be family of size 
$f_{(1,0,1,0)}(n)$ satisfying $(1,0,1,0)$-intersection pattern 
modulo $2$ 
and let 
$A \in \F$. Note that if $A = [n]$ then we must have $|\F| = 1$.
So we may assume $|A| \leq n-2$. The  
\nameref{lemma:trace} then gives 
$|\F| - 1 = |\F_{A}| \leq f_{(0,1)}(n-2) =  n-2$ as desired.

\subsection{The \texorpdfstring{$(0,0,1,0)$}{} Problem}\label{subsec:0010}
A family satisfying $(0,0,1,0)$-intersection pattern modulo $2$ also satisfies  
$(0,0,1)$-intersection pattern modulo $2$. Thus 
$f_{(0,0,1,0)}(n) \leq f_{(0,0,1)}(n)$. 
However, this is best possible as \autoref{construction:001_family}, the extremal 
constructions of families satisfying $(0,0,1)$-intersection pattern modulo $2$, is also a family satisfying 
$(0,0,1,0)$-intersection pattern modulo $2$. 
\subsection{The \texorpdfstring{$(0,1,0,0)$}{} Problem}\label{subsec:0100} 
We show that $f_{(0,1,0,0)}(n) \sim \sqrt{2n}$. 
Let 
    $b_2(n) := \max\left\{ \binom{j}{2} \ :  \binom{j}{2} \leq n \right\}$ be the largest triangular 
    number that is at most $n$. 
    We first show 
    $f_{(0,1,0,0)}(n) \leq 
    \sqrt{2b_2(n)} + 1 \leq 
    \sqrt{2n} + 1$. 
    Let $\F$ satisfy $(0,1,0,0)$-intersection pattern modulo $2$, be of maximum size, and have ground 
    set $[n]$. Abusing notation slightly, we will 
    write the set of pairwise intersections from 
    $\F$ as 
    \[
        \binom{\F}{2} := \{F_1 \cap F_2 \ : \ F_1,F_2 \in \F, F_1\neq F_2\}.
    \]
    Notice that $\binom{\F}{2}$ satisfies $(1,0)$-intersection pattern modulo $2$ with  
    ground set $[n]$. 
    Since $F_1 \cap F_2 \neq F_3 \cap F_4$ whenever $\{F_1,F_2\} \neq \{ F_3,F_4\}$ as $\F$ satisfies $(0,1,0,0)$-intersection pattern modulo $2$,  
    \[
        \binom{f_{(0,1,0,0)}(n)}{2} = \binom{|\F|}{2}
        = \left| \binom{\F}{2} \right| 
         \leq b_2(n) \leq n. 
    \]
    Rearranging gives us the desired upper bound. For the lower bound we have the below construction. 
\begin{const}\label{0100construction}
    For $n$ odd, define 
    $\A = \{A_1,\dots, A_n \}$ where 
    $A_i = \{e \in \binom{[n]}{2} \ : \ i \in e\}$. This family satisfies the $(0,1,0,0)$-intersection pattern modulo $2$. 
    \begin{figure}[h!]
    \centering
    \begin{tikzpicture}[thick,
  fsnode/.style={circle, fill=black, inner sep = 2pt},
  ssnode/.style={circle, fill=red, inner sep = 2pt},
  esnode/.style={circle, draw, inner sep = 2pt}
]
\begin{scope}[xshift = 0cm, yshift = 0cm]
        \foreach \i in {1,2,3,4,5,6,7}{%
        \node[fsnode] (\i) at (\i*51.4: 2cm) {};
        }
        \foreach \j in {1,2,3,4,5,6}{%
            \foreach \k in {1,2,3,4,5,6}{%
                \draw[-, dotted] (\j) to (\k) ;
        } 
        }
        \node[label = {[xshift = .5em, yshift = -.5em] $i$}] (l) at (0: 2cm) { };
        \draw[-,  very thick] (7) to (1);
        \draw[-,  very thick] (7) to (2);
        \draw[-,  very thick] (7) to (3);
        \draw[-,  very thick] (7) to (4);
        \draw[-,  very thick] (7) to (5);
        \draw[-,  very thick] (7) to (6);
    \end{scope}
    \end{tikzpicture}
    \caption{The sets in \protect\autoref{0100construction} for $n=7$}
    \end{figure}
\end{const}
The above construction 
yields the lower bound $f_{(0,1,0,0)}(n) 
\geq \sqrt{2n}$ when $n$ is of the 
form $\binom{j}{2}$ for some positive integer $j$. 
\par
To show that $f_{(0,1,0,0)}(n) \sim \sqrt{2n}$, we note first that the gap between 
$\binom{k}{2}$ and $\binom{k-2}{2}$ is $2k-3$. So for a general ground set $[n]$ there are  
at most $4\sqrt{n}$ of the elements of the ground set 
that are isolated vertices when applying \autoref{0100construction}.
So, 
$\sqrt{2(n - 4\sqrt{n})} \leq f_{(0,1,0,0)}(n) \leq \sqrt{2n} + 1$. 
It follows that $f_{(0,1,0,0)}(n) \sim \sqrt{2n}$, as desired.
\subsection{The \texorpdfstring{$(0,0,1,1)$}{} Problem}\label{subsec:0011}
We obtain an upper bound of 
    $f_{(0,0,1,1)}(n) 
    \leq b_2(n) \leq \sqrt{2n} + 1$, 
by using a similar argument
as in 
the beginning of 
\cref{subsec:0100}.
As for the construction, it is almost the same as \cref{0100construction} but with a 
small adjustment of adding an isolated vertex to each set and adjusting the parity of 
the vertex set accordingly.
\begin{const}
    For $n$ even,
    define 
    $\A = \{A_1,\dots, A_{n}\}$ where 
    $A_i := \{e\in \binom{[n]}{2} \ : \ i \in e\} \cup \{n+1\}$. 
    This family satisfies $(0,0,1,1)$-intersection pattern 
    modulo $2$. 
\end{const}
Using similar reasoning as in the last paragraph of $\cref{subsec:0100}$, we get $f_{(0,0,1,1)}(n) \sim \sqrt{2n}$. 

\subsection{The \texorpdfstring{$(0,0,0,1)$}{} Problem}\label{subsec:0001}
We show that $f_{(0,0,0,1)}(n) \sim \sqrt{2n}$. 
To prove that $f_{(0,0,0,1)}(n) \leq \sqrt{2n}+1$, we first give a lower bound for  
$f_{(0,1,0,0)}(n)$ in 
terms of $f_{(0,0,0,1)}(n)$ using the \nameref{lemma:PartitionSumLemma}. 
We then combine this 
with the upper bound for $f_{(0,1,0,0)}(n)$ we computed in \cref{subsec:0100}.

\par 
Indeed, first recall that  
$f_{(0,1,0,0)}(n) \leq \sqrt{2n}+1$ by \cref{subsec:0100}. 
On the other hand, we notice that 
    $(0,1,0,0) = (0,0,0,1) + (0,1,0,1)$ in $\mathbb{F}_2^4$.  
Thus, by the \nameref{lemma:PartitionSumLemma},   
\[
    f_{(0,1,0,0)}(n) \geq \min\Big\{f_{(0,1,0,1)}(\lfloor 2\sqrt{n}\rfloor), f_{(0,0,0,1)}(n- \lfloor 2\sqrt{n}\rfloor )\Big\}.
\]
If $f_{(0,1,0,1)}(\lfloor 2\sqrt{n}\rfloor ) \leq f_{(0,0,0,1)}(n-\lfloor 2\sqrt{n}\rfloor)$,
then 
\begin{align*}
    \sqrt{2n} + 1 &\geq f_{(0,1,0,0)}(n) \geq \min\Big\{f_{(0,1,0,1)}(\lfloor 2\sqrt{n} \rfloor ), 
    f_{(0,0,0,1)}(n- \lfloor 2\sqrt{n}\rfloor )\Big\} 
    = f_{(0,1,0,1)}(\lfloor 2\sqrt{n} \rfloor ) 
    \geq \lfloor 2\sqrt{n} \rfloor - 1,
\end{align*}
a contradiction for all $n\geq 25$. Thus 
\begin{align*}
\sqrt{2n} + 1 \geq f_{(0,1,0,0)}(n) 
\geq \min\Big\{f_{(0,1,0,1)}(\lfloor 2\sqrt{n}\rfloor ), f_{(0,0,0,1)}(n- \lfloor 2\sqrt{n} \rfloor )\Big\} 
= f_{(0,0,0,1)}(n-\lfloor 2\sqrt{n}\rfloor ). 
\end{align*}
For every $\epsilon > 0$, there is an $n_0$ such that for all $n\geq n_0$ we have 
$(1 - \epsilon)n < n - 2\sqrt{n}$. Thus  
\[
    f_{(0,0,0,1)}\Big( \lfloor (1-\epsilon)n\rfloor \Big) \leq f_{(0,0,0,1)}(n - \lfloor 2\sqrt{n}\rfloor ) \leq \sqrt{2n} + 1, 
\]
since $f_{(0,0,0,1)}(n)$ is nondecreasing. This completes 
the proof of the upper bound. 
\begin{const}\label{0001construction}
For $n$ odd, 
define $\A = \{A_1,\dots, A_{n}\}$ where 
$A_i := \Big([n] \setminus \{i\}\Big) \cup 
\{e \in \binom{[n]}{2} \ : \ i \in e\}$.
    \begin{figure}[h!]
    \centering
    \begin{tikzpicture}[thick,
  fsnode/.style={circle, fill=black, inner sep = 2pt},
  ssnode/.style={circle, fill=red, inner sep = 2pt},
  esnode/.style={circle, draw, inner sep = 2pt}
]
\begin{scope}[xshift = -2cm, yshift = 0cm]
         \foreach \i in {1,2,3,4,5,6,7}{%
        \node[fsnode] (\i) at (\i*51.43: 2cm) {};
        }
        \foreach \j in {1,2,3,4,5}{%
            \foreach \k in {1,2,3,4,5,6}{%
                \draw[-, dotted] (\j) to (\k) ;
        } 
        }
        \node[label = {[xshift = .5em, yshift = -.5em] $i$}] (l) at (0: 2cm) { };
        \draw[-,  very thick] (7) to (1);
        \draw[-,  very thick] (7) to (2);
        \draw[-,  very thick] (7) to (3);
        \draw[-,  very thick] (7) to (4);
        \draw[-,  very thick] (7) to (5);
        \draw[-,  very thick] (7) to (6);
    \end{scope}
    \begin{scope}[xshift = 7cm, yshift = 0cm]
        \draw[very thick] (-5,0) -- (0,0) ;
        \node[esnode, label={[xshift=0em, yshift=0] \small$i$}] (e) at (-1.5,0) { };
        \draw[very thick] (-5,.2) -- (-5,-.2);
        \draw[very thick] (-5,.2) -- (-5,-.2);
        \draw[very thick] (0,.2) -- (0,-.2);
        \node[label = {[xshift = 0em, yshift =0] $1$}] (leftend) at (-5,0) { };
        \node[label = {[xshift = 0em, yshift =0] $7$}] (leftend) at (0,0) { };
    \end{scope}
    \end{tikzpicture}
    \caption{The sets in \protect\autoref{0001construction} for $n = 7$}
    \end{figure}
    \par 
    This construction 
    satisfies $(0,0,0,1)$-intersection pattern modulo $2$.
\end{const}
The above construction gives a lower bound of $2k+1$ when the ground 
set has size $\binom{2k+1}{2} + 2k+1$. As before, when the ground set 
$[n]$ does not have this size, we need only set aside at most $4\sqrt{n}$ vertices 
from the ground set $[n]$. This gives us 
a lower bound of $f_{(0,0,0,1)}(n) \geq \sqrt{2(n - 4\sqrt{n})}$ for general $n$. By the same reasoning as the last paragraph of 
\cref{subsec:0100} we get $f_{(0,0,0,1)}(n) \sim \sqrt{2n}$.
\subsection{The \texorpdfstring{$(0,1,1,0)$}{} Problem}\label{subsec:0110}
The upper and lower bound are almost exactly the same as that proved for $f_{(0,0,0,1)}(n)$ 
in \cref{subsec:0001}. Indeed 
 $(0,0,1,1) = (0,1,1,0) + (0,1,0,1)$ in $\mathbb{F}_2^4$.  
 We then combine the \nameref{lemma:PartitionSumLemma} and the 
 bounds for $f_{(0,1,1,0)}(n)$ and $f_{(0,1,0,1)}(n)$
 in the same 
 way as in \cref{subsec:0001}. So for every $\epsilon > 0$, 
 there is an $n_0$ such that for all $n\geq n_0$,
\[
f_{(0,1,1,0)}\Big( \lfloor (1-\epsilon)n \rfloor \Big) \leq 
 f_{(0,1,1,0)}(n - \lfloor 2\sqrt{n}\rfloor ) \leq \sqrt{2n} + 1.
\]
The construction is similar to \autoref{0001construction}
but with an appropriate parity adjustment.
\begin{const}
For $n$ even,  
define $\A = \{A_1,\dots, A_{n}\}$ where 
$A_i := \Big([n] \setminus \{i\}\Big) \cup 
\{e \in \binom{[n]}{2} \ : \ i \in e\}$.
This family  
    satisfies $(0,1,1,0)$-intersection pattern modulo $2$. 
\end{const}

\subsection{Proof of \cref{thm:genk}}
Let $\alpha \in \ff{k}$. We may assume that $(\alpha_1, \alpha_2, \alpha_3, \alpha_4)$ is one of  
$(0,0,0,0)$, $(1,0,0,0)$, $(1,0,1,0)$ or $(0,0,1,0)$ as otherwise $f_\alpha(n) \leq f_{(\alpha_1, \alpha_2,\alpha_3, \alpha_4)}(n) = O(\sqrt{n})$. The analogous statement holds for $\alpha + \mathbf{1}$ 
by the
\nameref{The Dual Lemma}.

If $(\alpha_1, \alpha_2, \alpha_3, \alpha_4)=(0,0,0,0)$, then as $\alpha \neq \vb{0}$, there exist a minimum $k \in \N$ so that $\alpha_k=1$. The result then follows by noting that $(\alpha_{k-3},\alpha_{k-2}, \alpha_{k-1}, \alpha_k) = (0,0,0,1)$ and the \nameref{lemma:trace}, it follows that $f_\alpha(n) = O(\sqrt{n})$. 

If $(\alpha_1, \alpha_2, \alpha_3, \alpha_4)=(1,0,0,0)$, then as $\alpha \neq (1,0,\ldots,0)$, there exist a minimum $k \in \N$ so that $\alpha_k=1$. The result then follows by noting that $(\alpha_{k-3},\alpha_{k-2}, \alpha_{k-1}, \alpha_k) = (0,0,0,1)$ and the \nameref{lemma:trace}, it follows that $f_\alpha(n) = O(\sqrt{n})$.

If $(\alpha_1, \alpha_2, \alpha_3, \alpha_4)=(1,0,1,0)$, then as $\alpha \neq (1,0,1,0 \ldots)$, there exists minimal $2k-1 \in \N$ so that $\alpha_{2k-1}=0$ or minimal $2k \in \N$ so that $\alpha_{2k}=1$. Then, in the first case we have that $(\alpha_{2k-4}, \alpha_{2k-3}, \alpha_{2k-2}, \alpha_{2k-1}) = (0,1,0,0)$ and the result follows by the \nameref{lemma:trace}. In the second case, we have $(\alpha_{2k-3}, \alpha_{2k-2}, \alpha_{2k-1}, \alpha_{2k}) = (1,0,1,1)$ and the result follows by the 
\nameref{lemma:trace} together with the 
\nameref{The Dual Lemma}. 

If $(\alpha_1, \alpha_2, \alpha_3, \alpha_4)=(0,0,1,0)$, then as $\alpha \neq (0,0,1,0,1, \ldots)$, there exists minimal $2k-1 \in \N$ so that $\alpha_{2k-1}=1$ or minimal $2k \in \N$ so that $\alpha_{2k}=0$. Then, in the first case we have that $(\alpha_{2k-4}, \alpha_{2k-3}, \alpha_{2k-2}, \alpha_{2k-1}) = (1,0,1,1)$ and the result follows by the \nameref{lemma:trace}. In the second case, we have $(\alpha_{2k-3}, \alpha_{2k-2}, \alpha_{2k-1}, \alpha_{2k}) = (0,1,0,0)$ and the result follows by the \nameref{lemma:trace} together with the \nameref{The Dual Lemma}.

The case where $(\alpha_1, \alpha_2, \alpha_3, \alpha_4)$ is one of the corresponding duals is similar. \qedhere
 
\section{\texorpdfstring{$3$}{}-wise modulo \texorpdfstring{$3$}{} Patterns}\label{sec:mod3problems}
Here we briefly sketch the bounds claimed in Table 3. 

\subsection{\texorpdfstring{$g_{(\star,0,0)}(n)$}{} Bounds}\label{subsec:star00}

For the upper bound, note that any $\F \subset 2^{[n]}$ which satisfies $(\star,0,0)$-intersection pattern modulo $3$, satisfies the classical $\Mod{3}$-oddtown rules and hence $|\F| \leq n$. The singletons show this is best possible and therefore $g_{(\star,0,0)}(n) = n$. 
\par

\subsection{\texorpdfstring{$g_{(0, \star, \star)}(n)$}{} Bounds}\label{subsec:0starstar} 
We claim that $g_{(0,\star,\star)}(n) = \Theta(n^2)$. The upper bound follows from the Deza-Frankl-Singhi Theorem \cite{DFS} as any family $\F$ which satisfies $(0,\star,\star)$-intersection pattern is necessarily a $(3, \{1,2\})$-intersecting family and hence $|\F| \leq 1+n + \binom{n}{2}$.
The lower bound follows by noting that the $3$-uniform star $\{A \in \binom{[n]}{3} \ : \ 1 \in A\}$ satisfies $(0,\star,\star)$-intersection pattern modulo $3$ and has size $\binom{n-1}{2}$.

\subsection{\texorpdfstring{$g_{(\star,0,\star)}(n)$}{} Bounds}\label{subsec:star0star}
For the upper bound, we quickly obtain $g_{(\star,0,\star)}(n) \leq n$ by noting any family $\F \subset 2^{[n]}$ that satisfies $(\star,0,\star)$-intersection pattern modulo $3$ satisfies the classical modulo $3$ oddtown rules and hence $|\F| \leq n$. For the lower bound, we consider the construction $\binom{[n]}{n-1}$ when $n \equiv 2 \pmod{3}$, $\binom{[n-1]}{n-2}$ when $n \equiv 1 \pmod{3}$ and $\binom{[n-2]}{n-3}$ when $n \equiv 0 \pmod{3}$. It is not hard to see that such families satisfies $(\star,0,\star)$-intersection pattern modulo $3$ and have sizes $n$, $n-1$, and $n-2$ respectively.

\par

\subsection{\texorpdfstring{$g_{(0,\star,0)}(n)$}{} Bounds}\label{subsec:0star0} 
We obtain $g_{(0,\star,0)}(n) \leq n$ from the \nameref{lemma:trace} with the classical modulo $3$ oddtown problem, and noting that $[n]$ cannot be in such a family. 
For the lower bound, consider $k \equiv 0 \pmod{3}$ and let $n = 2k$. 
The family $\A = \{A_1, \ldots, A_k\}$ where 
\[
    A_i = ([k] \setminus i) \cup \{k+i\} \quad (i \in [k]),
\]
satisfies $(0,\star,0)$-intersection pattern modulo $3$ since 
its pairwise intersections have size 
$k-2 \equiv 1 \pmod{3}$ and its $3$-wise intersections have 
size $k-3 \equiv 0 \pmod{3}$. 

\subsection{\texorpdfstring{$g_{(\star,\star, 0)}(n)$}{} Bounds}\label{subsec:starstar0} 
For the lower bound, we consider the construction $\binom{[n]}{n-1}$ when $n \equiv 0 \pmod{3}$, $\binom{[n-1]}{n-2}$ when $n \equiv 1 \pmod{3}$ and $\binom{[n-2]}{n-3}$ when $n \equiv 2 \pmod{3}$. It is not hard to see that such families satisfies $(\star,0,\star)$-intersection pattern modulo $3$ and have sizes $n$, $n-1$, and $n-2$ respectively. In fact, these lower bounds are all best possible as we 
will now show. 
\par

Let $\F \subset 2^{[n]}$ satisfy $(\star,\star,0)$-intersection pattern modulo $3$. Fix any $F \in \F$. The \nameref{lemma:trace}, together with the classical modulo $3$ oddtown problem, show that $|\F| \leq |F|+1$. We therefore recover the desired result by noting that there exists $F \in \F$ so that $|F| \leq n-1$ when $n\equiv 0 \pmod{3}$, $|F| \leq n-2$ when $n \equiv 1\pmod{3}$, $|F| \leq n-3$ when $n\equiv 2 \pmod{3}$.

\subsection{\texorpdfstring{$g_{(0, 0, \star)}(n)$}{} Bounds}\label{subsec:00star}
The upper bound of $g_{(0, 0, \star)}(n) = O(n^2)$ follows by the \nameref{lemma:trace} and the Deza-Frankl-Singhi Theorem \cite{DFS}. The lower bound $g_{(0, 0, \star)}(n) = \Omega(n)$ comes from the following construction: 

Let $n=9k-3$ and $\{e_1, e_2, \ldots, e_{3k-1}\}$ be a graph perfect matching on $[3k,9k-3]$ (i.e. $|e_i|=2$ and $e_i \cap e_j = \emptyset$). Then, for $1 \leq i \leq 3k-1$ take $A_i = ([3k-1] \setminus i) \cup e_i$. Then $|A_i| = (3k-2)+2=3k$ and $|A_i \cap A_j| = (3k-3)$ and $|A_i \cap A_j \cap A_l| = 3k-4$ for $i,j,l$ distinct. Therefore $\A = \{A_1, \ldots, A_{3k-1}\}$ satisfies $(0,0,\star)$-intersections pattern modulo $3$ and has size roughly $n/3$.

\section{Concluding Remarks}
In this paper, we studied families satisfying 
$\alpha$-intersection patterns modulo $2$ for 
$\alpha \in \mathbb{F}_2^3$ and $\alpha \in \mathbb{F}_2^4$ and $\alpha$-intersection patterns modulo $3$ for $\alpha \in \{0,\star\}^3$. As with many problems in this area, 
two interesting directions to consider are  
the $r$-uniform variants (where we further require $\F \subseteq \binom{[n]}{r}$) and the modulo $\ell$ variants for general positive integers $\ell \in \N$.

\par
The $r$-uniform $\alpha=(0,0)$ intersection pattern modulo $\ell$ problem was determined (for large $n$) in \cite{FT2}, extending the more general Deza-Erd{\H o}s-Frankl bound \cite{DEF} to the case when $n \not\equiv 0 \pmod{\ell}$. This result naturally extends to the $r$-uniform $(0,\ldots,0)$-intersection pattern modulo $2$ problem, and it would be interesting to explore other $\alpha$-intersection patterns modulo $2$ in the $r$-uniform setting.

\par
For problems with intersection conditions modulo $\ell$ for an arbitrary $\ell \in \N$, a generalization of our constructions may be fruitful. To this end, consider the following construction: 

\begin{const}
Let $n,m \geq 1$ be integers and let $\mathcal{M} = \{ M_1, \ldots, M_n\}$ consist of $n$ pairwise disjoint sets of size $m$ on ground set $[n+1,n(m+1)]$. Define $A_i = ([n] \setminus \{i\}) \cup M_i$ and consider the family $\A_{n,m} = \{A_1, \ldots, A_n\}$. \end{const}

Observe that \cref{construction:110_family} essentially corresponds to $\A_{n,1}$ and the construction from \cref{subsec:00star} corresponds to $\A_{n,2}$. Further, note that the intersection sizes from $\A_{n,m}$ are as follows: 

\begin{itemize}
    \item $|A_i| = n+m-1$ for all $A_i \in \A_{n,m}$ 
    \item $|A_i \cap A_j| = n-2$ for all distinct $A_i, A_j \in \A_{n,m}$ 
    \item $|A_i \cap A_j \cap A_k | = n-3$ for all distinct $A_i, A_j, A_k \in \A_{n,m}$ 
\end{itemize}

For all $\ell \geq 1$ and all intersection patterns $\alpha \in \{0,\star\}^3$ where $\alpha \neq (0,0,0)$ and $\alpha \neq (\star,\star, \star)$, there exists a choice of $n,m$ for which $\A_{n,m}$ satisfies $\alpha$-intersection pattern modulo $\ell$. This shows that there exists a linear size construction of a family which satisfies $\alpha$-intersection pattern modulo $\ell$ for all $\alpha \in \{0,\star\}^3$. However, as in the case when $\ell = 3$, when $\alpha = (0,0,\star)$ and $\ell$ is prime, the corresponding upper bound from the Deza-Frankl-Singhi Theorem \cite{DFS} is of the order $\Theta (n^{\ell-1})$. It would be interesting to improve this gap from either direction.

\par 
For $k \geq 3$, we proved that there are only eight vectors
$\alpha \in \mathbb{F}_2^k$ with $f_\alpha(n) = \Omega(n)$. Moreover, the remaining $2^k-8$ vectors $\alpha \in \ff{k}$ satisfies $f_\alpha(n) = O(\sqrt{n})$, which leads naturally to the following: 
\begin{question}\label{question} 
Does there exists an absolute constant $C$ such that, for all $k\geq 1$, there are at most $C$ vectors $\alpha \in \mathbb{F}_2^k$ satisfying $f_\alpha(n) = \Omega(\sqrt{n})$? 
\end{question}
\par

\par
Since the initial arxiv posting of this preprint, there has been exciting progress towards \cref{question}. Wei, Zheng, and Ge \cite{WZG} answered \cref{question} in the affirmative showing that one may take $C=2^7$. It remains open to determine the optimal constant $C$.  

\vspace{5mm}
\par
{\bf Acknowledgements:} We are very
grateful to an anonymous referee for detailed comments which helped improve the paper. The second author would like to thank an anonymous referee from \cite{OV} for mentioning a particular case of Construction \ref{construction:001_family} as well as Xizhi Liu and Sayan Mukherjee for some initial conversations on these problems.

\bibliographystyle{plain}
\bibliography{bib}

\begin{appendices}
\section{ }\label[secinapp]{sec:appendix} 
\begin{appendixthm}[Folklore]\label{dual-eventown-canoncial}
In this appendix, we will prove that
\[
    f_{(1,1)}(n) = 
    \begin{cases}
        2^{\frac{n-2}{2}} & \text{$n$ even} \\
        2^{\frac{n-1}{2}} & \text{$n$ odd}.
    \end{cases}
\]
\end{appendixthm}
\begin{proof}
The extremal family consists of adding an auxiliary element to a canonical eventown construction. Let $2k+1 \in \Z$ be an odd integer and set $e_i=\{2i-1,2i\}$ for $i \in [k]$. Define 
\[ \F_k = \Bigg\{  \{2k+1\} \cup \; \bigcup\limits_{i \in I} e_i \; : I \in 2^{[k]} \Bigg\}.  \]
Observe that for $n=2k+1$ odd, $\F_k \subset 2^{[n]}$ satisfies $(1,1)$-intersection pattern modulo $2$ and for $n=2k$ even, $\F_{k-1} \subset 2^{[n-1]} \subset 2^{[n]}$ satisfies $(1,1)$-intersection pattern modulo $2$. Moreover, both constructions have the desired number of sets, establishing the lower bound.

For the upper bound, we proceed in a manner similar to the classical eventown proof. Let $\F \subset 2^{[n]}$ satisfies $(1,1)$-intersection pattern modulo $2$ and set $V \subset \ff{n}$ to be the corresponding characteristic vectors. Define $W = \Span(V)$ and consider the following map  $\phi: W \to \mathbb{F}_2$ where $w \mapsto w\cdot w$.  This is a surjective homomorphism, so 
$W/\ker\phi \cong \mathbb{F}_2$ giving us that  
$ |\ker\phi| = \frac12 |W| $. Moreover, as $\F$ satisfies $(1,1)$-intersection pattern modulo $2$, $\ker\phi \cap V = \emptyset$, which gives $|\F| = |V| \leq \frac12 |W|$. The proof is finished by the below claim: 

\begin{appendixclaim}\label{dimW upper bound}
    $\dim(W) \leq (n+1)/2$. Hence, for $n$ even, $\dim(W) \leq n/2$.
\end{appendixclaim}
\begin{proof}[Proof of \autoref{dimW upper bound}]
Let $B := \{v_1,\dots, v_m\}$ be a maximal linearly independent subset of $V$. We now consider the  
set $B_0 := \{v_1 + v_2, v_1 + v_3,\dots, v_1 + v_m\}$, and will show this is a linearly independent subset of $W^\perp$. 

Indeed, for any $v \in V$, $v \cdot (v_1 + v_i) = 1 + 1 = 0 \pmod{2}$ for all $i \in \{2,\dots, m\}$ since $V$ consists of the characteristic vectors of $\F$.  Since $W = \Span(V)$, it follows that $w \cdot (v_1 + v_i) = 0 \pmod{2}$ for all $w \in W$ and for all $i \in \{2, \ldots, m\}$. Therefore, $B_0 \subseteq W^\perp$.

To see linear independence, observe that we may rewrite any non-trivial combination of scalars $\lambda_2,\dots, \lambda_m \in \mathbb{F}_2$ as follows: 
\begin{align*}
0 = \sum_{i=2}^m \lambda_i (v_1 + v_i) =  
        \left(\sum_{i=2}^m \lambda_i\right) v_1 + 
        \sum_{i=2}^m \lambda_i v_i. 
\end{align*}
    
Thus, since $\{v_1,\dots, v_m\}$ is linearly independent subset of $W$, it follows that $\lambda_i = 0$ for all $i$. This means that $B_0 \subseteq W^\perp$ is linearly independent and therefore $\dim(W^\perp) \geq m-1$. Hence, 
\[
        n = \dim(W) + \dim(W^\perp) \geq 2m - 1 
        \implies m \leq \frac{n + 1}{2}. \qedhere
\]
\end{proof}
This completes the proof. 
\end{proof}
\end{appendices}

\end{document}